\documentclass[11pt]{article}

\usepackage{amsmath,amsthm,amssymb,color}
\usepackage{graphicx}
\usepackage{tikz-cd}

\parskip=\medskipamount
\def\hlift#1{#1^{\scriptscriptstyle{\mathrm{H}}}}
\def\vlift#1{#1^{\scriptscriptstyle{\mathrm{V}}}}

\def\fpd#1#2{{\displaystyle\frac{\partial #1}{\partial #2}}}

\def\hlift#1{#1^{\scriptscriptstyle{\mathrm{H}}}}
\def\vlift#1{#1^{\scriptscriptstyle{\mathrm{V}}}}

\def\R{\mathbb{R}}

\def\onehalf{{\textstyle\frac12}}

\font\frak=eufm10 scaled\magstep1
\def\goth#1{\hbox{{\frak #1}}}

\def\cinfty#1{C^{\scriptscriptstyle\infty}(#1)}
\def\vectorfields#1{\goth{X}(#1)}

\def\sode{{\sc sode}}
\def\sodes{{\sc sode}s}

\def\T{{\mathbf T}}

\theoremstyle{plain}
\newtheorem{theorem}{Theorem}
\newtheorem{lemma}{Lemma}
\newtheorem{proposition}{Proposition}
\newtheorem{definition}{Definition}

\begin{document}

\title{Jacobi fields and conjugate points for a projective class of sprays}

\author{ S.\ Hajd\'u$^\dagger$ and T.\ Mestdag$^{\dagger\,\ddagger}$ \\[2mm]
{\small $^\dagger$ Department of Mathematics,  University of Antwerp,}\\
{\small Middelheimlaan 1, 2020 Antwerpen, Belgium}\\[2mm]
{\small $^\ddagger$ 
Department of Mathematics: Analysis, Logic and Discrete Mathematics, Ghent  University}\\
{\small Krijgslaan 281, 9000 Gent, Belgium}\\[2mm]
{\small Email: sandor.hajdu@uantwerpen.be, tom.mestdag@uantwerpen.be} 
}

\date{}

\maketitle

\begin{abstract}
We investigate Jacobi fields and conjugate points in the context of sprays. We first prove that the conjugate points of a spray remain preserved under a projective change. Then, we establish conditions on the projective factor so that the projectively deformed spray meets the conditions of a proposition that ensures the existence of conjugate points. We discuss our methods by means of illustrative examples, throughout the paper.
\vspace{3mm}


\vspace{3mm}


\end{abstract}

\section{Introduction}

Jacobi fields and conjugate points play an essential role in the study of Riemannian and Finsler manifolds. 
Jacobi fields can be thought of as vector fields along a geodesic that measures the infinitesimal variation of a 1-parameter family of geodesics.
Two points along a geodesic are conjugate if
there exists a non-trivial Jacobi field along that curve that vanishes on both points. It is clear that both the absence and the existence of conjugate points are of interest and that conjugate points, for this reason, have always been  investigated extensively in the literature (see, for instance,  \cite{gallier2019differential,shen2013differential} or \cite{lee2018introduction} to mention just a few references.)

The geodesic equations of both Riemannian and Finsler metrics are essentially coupled systems of second-order ordinary differential equations (\sodes\, from now on).
In the context of \sodes\, the role of the curvature and the Levi-Civita connection is played by the so-called Jacobi endomorphism $\Phi$ and the covariant dynamical derivative $\nabla$ (see Section~2 for their definition, and for most of the preliminaries).  In a recent paper \cite{hajdumestdag1} we have discussed conjugate points for \sodes\, and in this paper  we will mainly rely on the following proposition:

\begin{proposition}\label{basicprop} \cite{hajdumestdag1} 
	Let $c$ be a base integral curve of a \sode\ $S$, through $m_0=c(0)$.  If 
	\begin{enumerate}
		\item[(1)] $\Phi$ has an eigenfunction $\lambda$ that remains  constant and strictly positive  along $c$, i.e.\ $\lambda_c(t)=\lambda_0>0$ for all $t$,
		\item[(2)]  there exists a non-vanishing  vector field  $V(t) \in \vectorfields c$ along $c$ that lies in $D_{\lambda_c}$, and which is such that $\nabla_c V(t) = 0$, 
	\end{enumerate}
	then the  points $c\big(\frac{k\pi}{\sqrt{\lambda}}\big)$  are conjugate to $m_0$.
\end{proposition}
The proof of this property is centered around the construction of the Jacobi field $J(t) =\sin(\sqrt{\lambda} t)V(t)$.

We have also shown that the existence of a parallel vector field (as called for in part (2)) can be guaranteed by requiring  that the so-called `bracket condition' is satisfied, $[\nabla \Phi,\Phi]=0$ (everywhere, or at least on the $\Phi$-eigendistribution of $\lambda$). This condition is quite familiar in the context of \sodes. For example, it is one of the conditions for a \sode\ to belong to ``Case II'' of the ``Inverse problem of Lagrangian mechanics'' (see e.g.\ \cite{crampin1999inverse}), and it is one of many conditions for a \sode\ to be ``separable'' (see e.g.\ \cite{martinez1993geometric}).

The geodesic equations of a Riemannian or Finsler metric are more specific than just \sodes. Because of the inherent homogeneity properties, they can be characterized by a special subclass of \sodes, namely the type of vector fields on the tangent manifold that are called sprays. On the other hand, the dynamical systems that are associated to sprays are of interest  in their own right. There exist sprays, even among those associated to linear connections, that are not the canonical spray of a Riemannian or Finsler metric (see e.g. \cite{BMM,Thompson} for the  canonical connection of a Lie group). Another example is the one that has been refered to as `Shen's circles' in \cite{crampin2012multiplier} (and was introduced by Shen in Section~4 of \cite{shen2013differential}). That spray is only the canonical spray of a Finsler function, after a projective change.

In this contribution we will focus on the added features of Proposition~\ref{basicprop} when the \sode\, is in fact a spray. For example, it is interesting in this context that a \sode\, is a spray if and only if $\dot c$ and $t\dot c$ are both Jacobi fields along each geodesic $c$. It is clear that there exist sprays that do not posses the bracket property. The spray of `Shen's circles' is an example in case. That does, however, not necessarily mean that we can not use the above proposition, because, in the context of sprays, we can utilize some extra freedom to enforce the bracket condition. 

A projective change of a spray $S$ is a spray $\tilde{S}=S-2P\Delta$, where $P$ is a positive homogeneous function and $\Delta$ is the Liouville vector field. It is well-known that a projective change of the spray does not affect the geodesics of the spray, when viewed as point sets. A projective change may however influence the diagonalization property of the Jacobi endomorphism (see Proposition~\ref{diag}). In Section~3 we will show that also conjugate points remain unaffected by a projective change (Theorem~\ref{main}). This means that, even when the spray $S$ does not satisfy the bracket condition, we may look for a projective change $P$ such that $\tilde S$ does (see e.g.\ Proposition~\ref{T0}).  At the end of Section~\ref{bracket} we focus our attention to isotropic sprays. This is an important subclass of sprays since, for example, all 2 dimensional sprays are isotropic. 

In Section~\ref{ex1} we discuss some examples. We have mainly focused on examples  that have circles as geodesics, in view of their accessible geometric interpretation (see also \cite{tabachnikov2004remarks}). Let $M$ be an open subset of the Euclidean plane, with Euclidean coordinates $(x,y)$, and let $({\dot x},{\dot y})$ be the corresponding fibre coordinates on the slit tangent bundle $T^\circ\! M$. In  \cite{crampin2013class} it is shown that, if a Finsler metric of Randers type whose Riemannian part is conformal to the Euclidean metric has only circles as geodesics, then its Riemannian part must be of constant Gaussian curvature. Up to a M\"obius transformation,  up to multiplication by a positive constant and the addition of an arbitrary total derivative, such a Finsler metric belongs to one of the following three one-parameter families of Randers metrics (with parameter $\tau$):
\begin{align*}
\mbox{(A)}\qquad F_\tau(x,y,{\dot x},{\dot y})&=\sqrt{{\dot x}^2+{\dot y}^2}+\tau(y{\dot x}-x{\dot y})
\\
\mbox{(B)}\qquad  F_\tau(x,y,{\dot x},{\dot y})&= \frac{\sqrt{{\dot x}^2+{\dot y}^2}+ \tau(y{\dot x}-x{\dot y})}{2(1+(x^2+y^2))} 
\\
\mbox{(C)}\qquad  F_\tau(x,y,{\dot x},{\dot y})&= \frac{\sqrt{{\dot x}^2+{\dot y}^2}+ \tau(y{\dot x}-x{\dot y})}{2(1-(x^2+y^2))}. 
\end{align*}
We use these three Finsler metrics as a running example throughout the paper.

Finally, in Section~\ref{ex2}, we point out that the methods we use all rely on the Picard-Lindel\"of theorem about existence and uniqueness of solutions of ordinary differential equations. As a consequence, the vector field $V(t)$ of Proposition~\ref{basicprop} is often only defined on a limited interval. We will use  one of the Finsler metrics above to show how an analysis of cut and conjugate points 
can, besides, help us to reach a conclusion about conjugate points.

\section{Preliminaries}

Let $M$ be a manifold. A {\em second-order differential	equation field} $S$ (from now on \sode, in short) is a vector field on the tangent manifold $TM$ with the property that all its integral curves are lifted curves $\dot c(t)$ of curves $c(t)$ in $M$ (the so-called base integral curves of $S$), i.e.\ they satisfy
\[
\ddot{c}(t) =S (\dot{c}(t)).
\]

The notion of a Jacobi field has been extended to \sode s in \cite{caririena1991generalized} (see also \cite{carinena2015jacobi}). The definition is based on the notion of a variational vector field.

A {\em 1-parameter family of integral curves of a vector field $Y\in\vectorfields{M}$} is a map $\zeta: ]-\epsilon,\epsilon[\times I\subset \R^2\to M$ such that for every $u\in]-\epsilon,\epsilon[$ the curve $\zeta_u:I\to M$, given by $\zeta_u(t):=\zeta(u,t)$ is an integral curve of $Y$. The vector field $Z$ along $\zeta_0$ defined by $Z(t)=\frac{\partial\zeta}{\partial u}(0,t)$ is said to be the {\em variational vector field} defined by the 1-parameter family.


In case the vector field $Y$ is a \sode\ $S$ - taking into account that the integral curves of $S$ are all lifted curves - the variation $\zeta(u,t)$ can be written as $\zeta(u,t)=\frac{\partial\gamma}{\partial t}(u,t)$, where $\gamma(u,t)$ is a 1-parameter family of base integral curves of the \sode. If we denote by $W(t)$ the variational vector field of the base family, that is if 
\[
W(t) = \frac{\partial \gamma}{\partial u}(0,t),
\]
then $Z(t)=W^c(t)$.

	Let $c$ be a base integral curve of a \sode\ $S$. A Jacobi field along $c$ is a vector field $J(t)$ along $c$, whose complete lift is a variational vector field along the integral curve $\dot c$ by integral curves of $S$.

Equivalenty, Jacobi fields can be characterized as solutions of a system of second-order ordinary linear differential equations, the so-called generalized Jacobi equations. In order to introduce them we first recall two important operators that one can associate to a \sode\ (but see \cite{sarlet2012} for a short review on this material). 

Let $\tau: TM \to M$ be the tangent bundle projection.
A {\em vector field along $\tau$} is a  map $X: TM \to TM$ with
the property that $\tau\circ X = \tau$. Any vector field $Y$ on $M$ induces a (so-called) `basic' vector field $X=Y\circ\tau$ along $\tau$. In natural coordinates 
$(x^i,{\dot x}^i)$ on $TM$, a vector field $X$ along $\tau$ can locally be expressed as
\[
X=X^i(x,{\dot x})\fpd{}{x^i}
\]
where $\fpd{}{x^i}$ are the coordinate vector fields on $M$, in their intepretation as vector fields along $\tau$. For example, we may always view the identity $v\mapsto v$ in a canonical way as a vector field along $\tau$. If we denote the correspoding section as $\T$, then
\[
\T={\dot x}^i \fpd{}{x^i}.
\]

A \sode\ is locally given by  \[
S={\dot x}^i\fpd{}{x^i}+f^i(x,{\dot x})\fpd{}{{\dot x}^i}.\]
Its {\em dynamical covariant derivative} $\nabla$ acts
as a derivative on  vector fields along $\tau$, in the sense that for
$f\in\cinfty{TM}$, 
\[
\nabla(fX) = f\nabla X + S(f)X.
\] 
The action of $\nabla$ on coordinate vector fields is 
\[
\nabla \fpd{}{x^j} = \Gamma_j^i \fpd{}{x^i} =-\onehalf \fpd{f^i}{{\dot x}^j} \fpd{}{x^i}.
\]
The {\em Jacobi endomorphism} $\Phi$ of $S$ is a type (1,1) tensor field along $\tau$, which in coordinates is given by
\[
\Phi\left( \fpd{}{x^j}\right)= \Phi^i_j \fpd{}{x^i} = \left(-\fpd{f^i}{x^j} - \Gamma^k_j\Gamma^i_k - S(\Gamma^i_j)\right)\fpd{}{x^i}.
\]

For any $v\in T_mM$, we may consider the endomorphism $\Phi_v: T_mM \to T_mM$. The collection of those for $v={\dot c}(t)$ can be interpreted as an operator $\Phi_c$ that maps vector fields along $c$ to vector fields along $c$. When $W(t)=W^i(t) \fpd{}{x^i}\bigg\rvert_{c(t)}$ is such a vector field, then
\[
\Phi_c(W(t)) = \Phi^i_j({\dot c}(t)) W^j(t) \fpd{}{x^i}\bigg\rvert_{c(t)}.
\] 
Likewise, by the relation
\[
\nabla_c W (t) = \left(\frac{d}{dt}W^i(t) + \Gamma^i_j({\dot c}(t)) W^j(t) \right)\fpd{}{x^i}\bigg\rvert_{c(t)}
\]
we define an operator $\nabla_c$ with the property
\[
\nabla_c(\mu(t)W(t)) = \dot \mu(t) W(t) + \mu(t) \nabla_c W(t), \qquad \mu \in \cinfty{\R}.
\]

In \cite{caririena1991generalized} it is shown (see e.g.\ also Theorem 2.7 of \cite{carinena2015jacobi}) that a vector field $J(t)$ along a base integral curve $c$  of a {\sode} $S$ is a Jacobi field if and only if it satisfies {\em the (generalized) Jacobi equation}
\[
\nabla_c\nabla_c J(t) + \Phi_c(J(t)) = 0.
\]

\begin{definition}
	Let $c$ be a base integral curve of a \sode\ $S$, through $m_0=c(0)$. If there exists a Jacobi field $J(t)$, not identically zero, with the property that $J(0) = J(t_1)=0$, then the point $m_1=c(t_1)$ is called a conjugate point of $m_0$ along $c$. 
\end{definition}

If a geodesic variation $\gamma (u,t)$ has fixed endpoints, say $\gamma(u,0)=m_0$ and $\gamma(u,t_1)=m_1$, its induced Jacobi vector field $\frac{\partial \gamma}{\partial u}(0,t)$ vanishes trivially on those points. These types of conjugate points will be called \textit{meeting points}. In general, not every Jacobi field are generated this way.  \\ \\
Nevertheless, to any given Jacobi field $J(t)$ with $J(0)=J(t_1)=0$ one can attach a variation with at least one fixed endpoint under the assumption of geodesically forward completeness. This means that every geodesic is assumed to be infinitely forward exdentable. This case, one can consider the geodesic variation using the exponential mapping of the spray: 

\[
\gamma(u,t)=\exp t(v+uw),
\]
where $\dot{c}(0)=v$ and $\nabla_c J(0)=w$. Each element of this family emanates from $c(0)$ and the correcponding variational vector field solves the same inital value problem as $J(t)$. (see Lemma 14.2.1 in \cite{shen2013differential}).

\section{Projective classes of sprays and conjugate points} \label{projclass}

Recall that a \sode\ $S$ is said to be a {\em spray} if $[\Delta,S]=S$, where $\Delta=\vlift\T$ is the Liouville vector field. 
Furthermore, a spray can equivalently be charaterized by the property $\nabla\T = 0$. In that case also $\Phi(\T) = 0$. This means that $\lambda=0$ is an eigenfunction of any spray.

Throughout the paper we will denote by $D_\lambda$ the eigendistribution of $\Phi$ corresponding to an eigenfunction $\lambda \in \cinfty{TM}$.
   For later reference, we start with a technical lemma.
\begin{lemma} Let $S$ be a spray.
Let $\lambda$ be an eigenfunction of $\Phi$, and $X\in D_{\lambda}$ a corresponding eigen vector field along $\tau$. Then, also $[\Delta,X^{\scriptscriptstyle{\mathrm{V}}}]_{\scriptscriptstyle{\mathrm{V}}}\in D_{\lambda}$.
\end{lemma}
\begin{proof}
Consider the operator $
\nabla_{\Delta} : \mathfrak{X}(\tau) \rightarrow \mathfrak{X}(\tau)
$
defined by
\[
\nabla_{\Delta}(X):=[\Delta,X^{\scriptscriptstyle{\mathrm{V}}}]_{\scriptscriptstyle{\mathrm{V}}}.
\] 
It is shown in \cite{szilasi2013connections} that, for a spray $S$,  $\nabla_{\Delta}(\Phi)=2\Phi$. This means that for all $X\in\mathfrak{X}(\tau)$,
\[
\nabla_{\Delta}(\Phi(X))-\Phi(\nabla_{\Delta}X)=2\Phi(X).
\] 
Suppose now that $X\in D_\lambda$. Then
\[
\nabla_{\Delta}(\Phi)=[\Delta,\lambda X^{\scriptscriptstyle{\mathrm{V}}}]_{\scriptscriptstyle{\mathrm{V}}}-\Phi([\Delta,X^{\scriptscriptstyle{\mathrm{V}}}]_{\scriptscriptstyle{\mathrm{V}}})=\Delta(\lambda) X+\lambda[\Delta,X^{\scriptscriptstyle{\mathrm{V}}}]_{\scriptscriptstyle{\mathrm{V}}}-\Phi([\Delta,X^{\scriptscriptstyle{\mathrm{V}}}]_{\scriptscriptstyle{\mathrm{V}}}).
\]
 In \cite{szilasi2013connections} it is also shown  that  $\Delta\lambda=2\lambda$, from which it follows that
 
 \[
 2\lambda X = 2\lambda X +\lambda[\Delta,X^{\scriptscriptstyle{\mathrm{V}}}]_{\scriptscriptstyle{\mathrm{V}}}-\Phi([\Delta,X^{\scriptscriptstyle{\mathrm{V}}}]_{\scriptscriptstyle{\mathrm{V}}}),
 \]
whence our claim.
\end{proof}

Throughout the paper we will make use of some well-known facts about geodesics, pregeodesics and projective changes. We will summarize some of them here, but for the proofs and more details on Lemma~\ref{affine} and Lemma~\ref{Lemma2} we refer  to \cite{szilasi2013connections} (Section~8.4).


\begin{lemma}\label{affine}
	Let $c:I\rightarrow M$ be a geodesic of a spray $S$ and consider the parameter transformation $\theta:\tilde{I}\rightarrow I,\phantom{a} \theta(s)=as+b$ for some $a,b\in\R$ with $a>0$, and $\tilde{I}$ an open interval. Then $\tilde c=c\circ\theta:\tilde{I}\rightarrow M$ is again a geodesic of $S$.
\end{lemma}
   
In case the \sode\ is a spray, the  coefficients $f^i$ satisfy
\[
f^i(x,\lambda {\dot x})=\lambda^2 f^i(x, {\dot x})
\]
for all $\lambda >0$. If the coordinate expression of the geodesic $c(t)$ of a spray is given by $x^i(t)$, then these functions satisfy
\[
\ddot{x}^i(t)=f^i(x(t),\dot{x}(t)),
\]
which holds also true for $\tilde c(s)$, with $x^i(s) =x^i(\theta(s))$, in view of the homogeneity property.

In what follows, we will often make use of the proposition that we have mentioned in the Introduction (Proposition~\ref{basicprop}). Let's consider the  class of Finsler metrics (B) (see the Introduction) by means of example.

{\bf Example.} We consider the geodesics of the Finsler function 
\[
F_\tau=\frac{\sqrt{\dot{x}^2+\dot{y}^2}+\tau(y\dot{x}-x\dot{y})}{2(1+x^2+y^2)}.
\]
One may verify that from all  circles that are centered around the origin, only those of radius $r=\Big\lvert{\frac{1}{\tau \pm \sqrt{\tau^2+1}}}\Big\rvert$ are  geodesics. For each value of $\tau$, the canonical spray of the corresponding Finsler function satisfies the bracket condition $[\nabla\Phi,\Phi] = 0$. We are therefore in a situation where we only need to check condition (1) of Proposition~\ref{basicprop}.

In case $\tau=0$ the metric is Riemannian and it corresponds to the standard  metric of the sphere after stereographic projection. Then, the above mentioned two circles coincide and they form the stereographic projection of the equator. It is well known that antipodal points are the only conjugate points on the sphere.  The non-zero eigenvalue of the Jacobi endomorphism along the geodesic in question has the constant value of $1$. In this case, one may verify that 
\[
V(t)=\cos(t)\frac{\partial}{\partial x}\bigg\rvert_{c(t)}+\sin(t)\frac{\partial}{\partial  y}\bigg\rvert_{c(t)}.
\] 
is a parallel eigenvector-field, defined for all values of $t$. 

When $\tau$ is different from zero, there are two different geodesic circles centered at the origin, but they still have the property that $\lambda$ has the constant value of $1$ along them. In this case, the Jacobi equation along these circles is independent of $\tau$ (and is in fact identical with the $\tau=0$ case). As a consequence, for each value of $\tau$, the canonical spray of the Finsler function admits conjugate points along these geodesics at the parameter value $\pi$ and its integer multipliers.

It is clear that we could apply Proposition~\ref{basicprop} in this example, because of the bracket property. In what follows, we will focus on the situation when the spray does not satisfy this condition.

\begin{definition} \label{Def1}
	Let $S$ be a spray on $M$. 
	\begin{itemize} \item A curve $c:I\rightarrow M$ is called a pregeodesic of $S$, if there exists a parameter transformation $\theta:\tilde{I}\rightarrow I$ with strictly positive derivative, so that $\tilde c =c\circ \theta$ is a geodesic of $S$. In that case, we say that the curve $c$ can be positively reparametrized to be a geodesic.
	\item A spray $\tilde{S}$ is projectively related to $S$ if there exists a (1-homogeneous) function $P$ on $TM$ such that $\tilde S=S-2P\Delta$.
	\end{itemize}
\end{definition}

The set of all sprays that are projectively related to each other is called a {\em projective class of sprays}. The following lemma provides a relation between the concepts of Definition~\ref{Def1}.

	%
%
%

\begin{lemma} \label{Lemma2}
Let $S$ and $\tilde S$ be two projectively related sprays, $\tilde{S}=S-2P\Delta$. If  $c(t):I\rightarrow M$ is a geodesic of $S$, then it is a pregeodesic of $\tilde S$. 
\end{lemma}
The proof is based on the fact that, 	for any  smooth function  $h:I\rightarrow \R$,
the differential equation
\[{\theta}''+(h\circ \theta)({\theta}')^2=0
\]
admits a smooth solution, satisfying the inital condition $\theta(0)=0$ and $\dot{\theta}(t)>0$. In case $h$ is $2P\circ {\dot c}$ that solution $\theta(s)$ can be used to show that $\tilde c(s)=c(\theta(s))$ is a geodesic of $S$.

The parameter transformation of Lemma~\ref{affine} (with $b=0$) fits within this context, since the choice $P=0$ leads to the differential equation $\theta''=0$. This means that, again for arbitrary $P$, the transformation $\theta(s)$ is not unique. For example, also $\hat\theta(s) = a\theta(s)$ (with $a>0$) will be an adequate transformation.

Our exploration of conjugate points for sprays is based on the following theorem.
\begin{theorem}\label{main} 	The conjugate points of a spray are preserved under a projective change.
\end{theorem}

\begin{proof}
Let us assume that the points $m_0=c(0)$ and $m_1=c(t_1)$ are conjugate points along the geodesic $c(t)$ of the spray $S$, for the  Jacobi field  $J(t)$  along $c(t)$. Assume that the Jacobi field $J$ comes from the 1-parameter family of solutions, $\gamma(u,t)$, with $\gamma(0,t)=c(t)$ and $\gamma(u,0)=m_0$. Then
\[
J(t)=\frac{\partial}{\partial u}\bigg\rvert_{u=0}\gamma(u,t).
\]

Consider a projective change of $S$ with projective factor $P$, that is $\tilde{S}:=S-2P\Delta$.
We construct a Jacobi field for $\tilde{S}$.

\textit{Step 1.} Since $\gamma_u(t)=\gamma(u,t)$ is a geodesic of $S$ for each $u$, it is a pregeodesic of $\tilde{S}$ in view of Lemma~\ref{Lemma2}. This means that we can find for each $u$ a parameter transformation $\theta_u(s)$ such that $\tilde{\gamma}_u(s):=\gamma_u(\theta_u(s))$ is a geodesic of $\tilde S$. The functions $\theta_u$ are solutions of 
\[{\theta}_u''+2(P\circ {\dot \gamma}_u\circ \theta_u)({\theta}_u')^2=0
\]
and satisfy the inital condition $\theta_u(0)=0$. The family $\tilde\gamma(u,s):=\tilde\gamma_u(s)$ that we find in this way is then a 1-parameter family of geodesics for $\tilde{S}$ around $\tilde\gamma(0,s)=\tilde\gamma_0(s)=c(\theta_0(s))=\tilde c(s)$. The initial condition ensures that all members of the reparametrized family start at the same point $\tilde\gamma(u,0)=c(\theta_u(0))=c(0)=m_0$.

\textit{Step 2.}
Since the members of the new family $\tilde\gamma_u(s)$ are identical with the members of the old one when consideres as point sets, we know that they all intersect the transversal line $\gamma_{t_1}(u)$, but possibly at different parameter values $s$.

In this step, we wish to ensure that the reparametrized family  reaches the line $\gamma_{t_1}(u)$ at a common parameter value $s_1$. For this reason we consider the parameter transformation $\theta_0$ that corresponds to $\gamma_0:=c$. Let us denote the real number $t_0^{-1}(t_1)$ by $s_1$. Since $t_0$ is a strictly increasing function, going through the origin, it follows that $s_1$ is positive, and we may consider the positive number  $a_u = \frac{t_1}{t_u(s_1)}$ for each $u$. The function  $\hat\theta(u,s):=a_u\theta_u(s)$ then has the property that $\hat\theta(u,0)=0$ and $\hat\theta(u,s_1)=t_1$ are constant in $u$. As we remarked before we now have a parameter transformation that maps the geodesic $\gamma_u(t)$ of $S$ to the geodesic $\hat{\gamma}_u(s):=\gamma(u, \hat\theta(u,s))$ of $\tilde S$. Since $a_0=1$, $\hat\theta(0,s)=\theta_0(s)$ and the  new family  $\hat\gamma(u,s):=\hat\gamma_u(s)$ is still centered around  $\hat\gamma(0,s)=\tilde c(s)$. Moreover, it has the property that $\hat\gamma(u,0)=c(\hat\theta(u,0))=c(0)=m_0$, for all $u$.


\textit{Step 3.} The Jacobi field of the family $\hat{\gamma}(u,s)$ of geodesics of $\tilde S$ is given by 
\begin{eqnarray*}
\hat{J}(s)&=&\frac{\partial}{\partial u}\bigg\rvert_{u=0}\Big(\hat{\gamma}(u,s)\Big)= \frac{\partial\gamma}{\partial u}(0,\hat\theta(0,s)) + \frac{\partial\gamma}{\partial t}(0,\hat\theta(0,s))   \frac{\partial \hat\theta}{\partial u}(0,s) \\&=& J(\theta_0(s))+\dot{c}(\theta_0(s)) \frac{\partial \hat\theta}{\partial u}(0,s).
\end{eqnarray*}

Since $\theta_0(0)=0$, $J(0)=0$ and $\hat\theta(u,0)=0$ (for all $u$), we have $\hat J(0)=0$. Likewise, from $\theta_0(s_1)=t_1$, $J(t_1)=0$ and $\hat\theta(u,0)=t_1$ (for all $u$), it follows that $\hat J(s_1)=0$. We may therefore conclude that $\tilde c(0) = c(\theta_0(0))=c(0)=m_0 $ and $\tilde c(s_1)=c(\theta_0(s_1))=c(t_1)=m_1$ are also conjugate points for the spray $\tilde S$.\end{proof}

We now relate Theorem~\ref{main} to Proposition~\ref{basicprop}. Say that a spray $S$ does not satisfy all the conditions of Proposition~\ref{basicprop}.  Since all sprays of the class have the same conjugate points, we can make use of the freedom in the projective factor $P$ to search for a spray $\tilde S = S-2P\Delta$ within the projective class of $S$ that does satisfy all the necessary conditions.

First, we recall (from e.g.\ \cite{BM} or \cite{szilasi2013connections}) how $\nabla$ and $\Phi$ change after a projective deformation. The Jacobi endomorphism of $\tilde S = S-2P\Delta$ is given by:
	\begin{equation} \label{Phichange}
		\tilde{\Phi}(X)=\Phi (X) +aX +b(X)\T,
	\end{equation}
	where $a=P^2-S(P)$ and $b(X)=3\hlift{X}(P)-P\vlift{X}(P)-\vlift{X}(S(P))$. Since for sprays $\Phi(\T)=\tilde\Phi(\T)=0$, it is easy to see that $b(\T)=-a$.

Likewise, the action of the dynamical covariant derivative of $\tilde S$ can be written as
	\begin{equation}\label{nabla}
	\tilde{\nabla}X=\nabla X-2P[\Delta,X^{\scriptscriptstyle{\mathrm{V}}}]_{\scriptscriptstyle{\mathrm{V}}}+\vlift{X}(P)\T -PX.
	\end{equation}

Throughout the paper we assume that $\Phi$ is diagonalizable, but this property is not necessarily preserved under a projective change. The following lemma provides us conditions under which $\tilde{\Phi}$ is diagonalizable. From $\Phi(\T)=0$ we may conclude that $\T$ is always an eigenvector with eigenvalue 0. We will denote the other eigenfunctions by $\lambda_2,\ldots \lambda_n$. Herein it is understood that some of them may be equal, and that some of them may be zero. 

\begin{lemma}\label{diag} 
	Suppose that the Jacobi endomorphism $\Phi$ of a spray $S$ is diagonalizable, and consider a projective change $\tilde{S}:=S-2P\Delta$. Then $\tilde{\Phi}$ is diagonalizable if and only if one of the following conditions are satisfied:
	
	\begin{enumerate}
			\item[(1)] $a\neq -\lambda_j$ for any $j$ in $2,\ldots, n$, 
	\item[(2)]  $a=-\lambda_j$ for some $j$ in $2,\ldots, n$ and $b(X)=0$ for all $X\in D_{\lambda_j}$.  
	\end{enumerate}

  In each of these cases,
	\begin{enumerate}
		\item[(1)] if $\lambda$ is an eigenfunction of $\Phi$ then $\tilde{\lambda}:=\lambda +a$ is an eigenfunction of $\tilde{\Phi}$,
		\item[(2)]  if $X \in D_{\lambda}$, then $\tilde{X}:=(\lambda +a)X+b(X)\T \in {\tilde D}_{\tilde\lambda}$.
	\end{enumerate} 
\end{lemma}

\begin{proof}
Since we assume that $\Phi$ is diagonalizable, we may fix an eigenbasis $(\T,X_2,\dots X_n)$. Let $(b_i)$ be the components of the one-form $b$, w.r.t this basis. The matrix of $\tilde{\Phi}$ w.r.t the eigenbasis of $\Phi$ is then
	\[
\begin{bmatrix} 
0 & 0 & \dots & 0\\
0 & \lambda_2 & \dots & 0 \\
\vdots & \vdots & \ddots & \vdots\\
0 &   0 &   \dots & \lambda_n & 
\end{bmatrix}
+
\begin{bmatrix} 
a & 0 & \dots & 0\\
0 & a & \dots & 0 \\
\vdots & \vdots & \ddots & \vdots\\
0 &   0 &   \dots & a & 
\end{bmatrix}
+
\begin{bmatrix} 
b_1 & b_2 & \dots & b_n\\
0  & 0 & \dots & 0 \\
\vdots & \vdots & \ddots & \vdots\\
0 &  0 &   \dots & 0 & 
\end{bmatrix}
.
\]
Given that $b_1=b(\T)=-a$, this matrix becomes
 \[
 \begin{bmatrix} 
 0 & b_2 & \dots & b_n\\
 0 & \lambda_2+a & \dots & 0 \\
 \vdots & \vdots & \ddots & \vdots\\
 0 &   0 &   \dots & \lambda_n+a & 
 \end{bmatrix}
 .
 \]
It is easy to see that  its eigenvalues are ${\tilde \lambda}_1=0$ and $\{{\tilde\lambda}_i=\lambda_i+a\}_{i=2..n}$. Moreover, one may readily check that 
\[
\tilde{X}_1=\T,
\qquad \tilde{X}_i=b(X_i)\T + (\lambda_i +a)X_i
\]
are eigenvectors for $\tilde\Phi$ for the eigenvalues ${\tilde \lambda}_1=0$ and $\tilde\lambda_i$, respectively. These vectors will form a new basis if and only if the  matrix 
\[
 \begin{bmatrix} 
1 & 0 & \dots & 0\\
b_1 & \lambda_2+a & \dots & 0 \\
 \vdots & \vdots & \ddots & \vdots\\
 b_2 &   0 &   \dots & \lambda_n+a & 
 \end{bmatrix}
 .
 \]
has non-vanishing determinant. When $a\neq -\lambda_i$ for any $i$ within $2,\ldots n$, this is the case. 

Assume now that $a=-\lambda_i$, for some eigenvalue $\lambda_i$ within $2,\ldots n$ that has multiplicity $m$ (which may or not be zero). In that case, the above vector fields do not form a basis, and we need to replace the $m$ vectors $\{{\tilde X}_{i+\alpha}\}_{\alpha=0,\ldots,m-1}$ by $m$ linear independent eigenvectors of $\tilde \Phi$ with new eigenvalue 0. In general, a vector $x_1\T+\sum_{\alpha=0}^{m-1} x_{i+\alpha} X_{i+\alpha} + \sum_{\beta} x_\beta X_\beta$ (where $\beta$ runs over the rest) will be such if all $x_\beta =0$ and $\sum_{\alpha=0}^{m-1} x_{i+\alpha} b(X_{i+\alpha}) =0$. This last equation can only deliver $m$ independent vectors if all $b(X_{i+\alpha}) =0$.\end{proof}

A projective factor for which $a =0$ is called a {\em weak Funk function} in \cite{szilasi2013connections}. Although it changes $\Phi$, it does not change its eigenvalues. Then, the diagonalizability of $\Phi$ develops, as follows. 

\begin{lemma}\label{diagfunk} 
	Suppose that the Jacobi endomorphism $\Phi$ of a spray $S$ is diagonalizable, and consider a projective change $\tilde{S}:=S-2P\Delta$ by a weak Funk function. Then $\tilde{\Phi}$ is diagonalizable if and only if one of the following conditions are satisfied:
	\begin{enumerate}
			\item[(1)] The eigendistribution $D_0$ of $\lambda=0$ is 1 dimensional, 
	\item[(2)]  $b(X)=0$ for all $X\in D_{0}$.  
	\end{enumerate}
  In each of these cases,
	\begin{enumerate}
		\item[(1)] if $\lambda$ is an eigenfunction of $\Phi$ then it is also an eigenfunction of $\tilde{\Phi}$,
		\item[(2)]  if $X \in D_{\lambda}$, then $\tilde{X}:=\lambda X+b(X)\T \in {\tilde D}_{\lambda}$.
	\end{enumerate} 
\end{lemma}

In the context of our method, eigenvectors corresponding to the zero eigenfunction play no role. For this reason one may relax the condition of diagonalizability on $\Phi$ by requiring that the algebraic and geometric multiplicity agree only for the non-zero eigenfunctions. Nevertheless, in what follows, we will assume that one of the conditions  of Lemma~\ref{diag} is satisfied, whenever we consider a spray and a projective change. In fact, all our examples fall into case (1) of Lemmma~\ref{diag}. One finds examples belonging to case  (2) for instance in \cite{elgendi2019metrizability}, where the authors investigate questions about the metrizability of sprays and holonomy invariant projective changes.

\section{The bracket property}\label{bracket} 

In \cite{hajdumestdag1}, we have shown that the existence of the parallel vector field in part  (2)  of Proposition \ref{basicprop} can be guaranteed by requiring the condition $[\nabla \Phi,\Phi]=0$ on the spray.  When restricted to an eigendistribution of $\Phi$, this condition can be characterized  as follows.

\begin{proposition}\cite{hajdumestdag1}\label{bracketequivalent} 
	 Let $\lambda$ be an eigenfunction of $\Phi$. The following statements are equivalent:
	\begin{enumerate}
		\item $[\nabla\Phi,\Phi](D_\lambda) = 0$,
		\item $\nabla\Phi(D_\lambda) \subset D_\lambda$,
		\item  $\nabla D_\lambda \subset D_\lambda$.
	\end{enumerate}
\end{proposition}

We investigate whether the bracket property can be achieved by applying a projective change:

\begin{proposition}\label{T0} 
	Consider a spray $S$ and a projective change by $P$ and let $\tilde{\lambda}=\lambda +a$ be a non-zero eigenfunction of $\tilde{\Phi}$. Then the condition $[\tilde{\nabla}\tilde{\Phi},\tilde{\Phi}]=0$ is satisfied on $\tilde{D}_{\tilde{\lambda}}$ if and only if
	\begin{equation}\label{T0equation}
		\Phi(\nabla X)-\lambda \nabla X=\bigg( -(\nabla b)+\bigg(-\frac{S(\lambda +a)}{\lambda +a}+P\bigg)b+ (\lambda+a)\vlift{d}(P)\bigg)(X)\T\quad \text{on}\phantom{a}D_{\lambda}.
	\end{equation}
\end{proposition}
\begin{proof}
		Let $\tilde{\lambda}$ be an eigenfunction of $\tilde{\Phi}$. From the proof of Lemma~\ref{diag} we know that $\lambda:=\tilde{\lambda}-a$ is a non-zero eigenfunction of $\Phi$, and that the set $\{\tilde{X}_i\}_{i=1..dim(D_{\lambda})}:=\{(\lambda+a)X_i+b(X_i)\T\}_{i=1..dim(D_{\lambda})}$ (with $X_i\in D_{\lambda}$) spans $\tilde{D}_{\tilde{\lambda}}$.  According to condition (3) in Proposition~\ref{bracketequivalent}, an equivalent condition for $[\tilde{\nabla}\tilde{\Phi},\tilde{\Phi}]=0$ to hold true is that
	\[
	\tilde{\Phi}(\tilde{\nabla}\tilde{X})=\tilde{\lambda}\tilde{\nabla}\tilde{X}
	\]
	for all $X\in D_{\lambda}$. Relation (\ref{nabla}) leads to
	\begin{eqnarray}
	\nonumber	\tilde{\Phi}\big(\nabla \tilde{X}-2P[\Delta,\tilde{X}^{\scriptscriptstyle{\mathrm{V}}}]_{\scriptscriptstyle{\mathrm{V}}}+\vlift{\tilde{X}}(P)\T -P\tilde{X}\big)&=&\tilde{\lambda}\big(\nabla \tilde{X}-2P[\Delta,\tilde{X}^{\scriptscriptstyle{\mathrm{V}}}]_{\scriptscriptstyle{\mathrm{V}}}+\vlift{\tilde{X}}(P)\T -P\tilde{X}\big)\\
\nonumber	\iff \tilde{\Phi}(\nabla \tilde{X})-\tilde{\lambda}2P[\Delta,\tilde{X}^{\scriptscriptstyle{\mathrm{V}}}]_{\scriptscriptstyle{\mathrm{V}}}-\tilde{\lambda}P\tilde{X}&=&\tilde{\lambda}\nabla \tilde{X}-\tilde{\lambda}2P[\Delta,\tilde{X}^{\scriptscriptstyle{\mathrm{V}}}]_{\scriptscriptstyle{\mathrm{V}}}+\tilde{\lambda}\vlift{\tilde{X}}(P)\T -\tilde{\lambda}P\tilde{X}\\
	 \nonumber	\iff \tilde{\Phi}(\nabla \tilde{X})&=&\tilde{\lambda}\nabla \tilde{X}+\tilde{\lambda}\vlift{\tilde{X}}(P)\T,
	\end{eqnarray}
	where in the first step we have used that $\T\in \tilde{D}_0$. We now use relation (\ref{Phichange}) to rewrite the left-hand side in terms of quantities corresponding to the starting spray. For computational reasons we evaluate the expression at $\frac{\tilde{X}}{\lambda+a}$ instead of $\tilde{X}$.
	\[
	\tilde{\Phi}\bigg(\nabla \frac{\tilde{X}}{\lambda +a}\bigg) =\tilde{\Phi}\bigg(\nabla X+S\bigg(\frac{b(X)}{\lambda +a}\bigg)\T \bigg)=\tilde{\Phi} (\nabla X)=\Phi(\nabla X)+a\nabla X +b(\nabla X)\T.
	\]
	
Doing the same on the right-hand side leads to
\begin{eqnarray*} &&
\hspace*{-2cm} \tilde{\lambda}\nabla \bigg(\frac{\tilde{X}}{\lambda +a}\bigg)+\bigg(\tilde{\lambda}\frac{\vlift{\tilde{X}}}{\lambda +a}\bigg)(P)\T\\ &=&(\lambda +a)\nabla\bigg(X+\frac{b(X)}{\lambda +a}\T\bigg)+(\lambda +a)\bigg(\vlift{X}(P)+\frac{b(X)}{\lambda +a}\Delta(P)\bigg)\T\\
\nonumber &=&(\lambda +a)\nabla X+(\lambda +a)S\bigg(\frac{b(X)}{\lambda +a}\bigg)\T+(\lambda+a)\vlift{X}(P)\T+b(X)P\T \\
\nonumber &=&(\lambda +a)\nabla X+S(b(X))\T-b(X)\frac{S(\lambda +a)}{\lambda +a}\T+(\lambda+a)\vlift{X}(P)\T+b(X)P\T.
\end{eqnarray*}
When we compare the two sides, we see that the property $\tilde{\Phi}(\tilde{\nabla}\tilde{X})=\tilde{\lambda}\tilde{\nabla}\tilde{X}
	$ is equivalent with 
\begin{eqnarray*}
\Phi(\nabla X)-\lambda \nabla X &=&\bigg(-b(\nabla X) +S(b(X))-b(X)\frac{S(\lambda +a)}{\lambda +a} +(\lambda+a)\vlift{X}(P)+b(X)P\bigg)\T \\
\nonumber &=& \bigg( -(\nabla b)(X)+\bigg(-\frac{S(\lambda +a)}{\lambda +a}+P\bigg)b(X)+ (\lambda+a)\vlift{d}(P)X\bigg)\T.
\end{eqnarray*}
\end{proof}

Proposition~\ref{T0} can be simplified when the starting spray $S$ already satisfies $[\nabla \Phi,\Phi]=0$ on $D_{\lambda}$. In such a case, the left-hand side of equation (\ref{T0equation}) vanishes. 

In view of the first condition in Proposition~\ref{basicprop}, we may reach an even simplier expression if $\tilde{\lambda}$ is a first integral of $\tilde{S}$. If that is the case, all geodesics are constant along $\tilde{\lambda}$ and the factor $-\frac{S(\lambda +a)}{\lambda +a}+P$ simplifies to $-4P+P=-3P$.  For instance, locally symmetric {\sodes} fall into that category (see \cite{hajdumestdag1}, Section 5 for details). The condition on the projective factor that guarantees this property can be calculated as follows.
\begin{lemma}\label{T1}
	Consider a spray $S$ and a projective change by $P$. Then, for any non-zero eigenfunction $\lambda$ of $\Phi$ the following two conditions are equivalent:
\begin{eqnarray}
\hspace*{-6cm}\mbox{\textit{1.}}&&  \tilde{S}(\tilde{\lambda})=0 \nonumber \\
\hspace*{-6cm}\mbox{\textit{2.}}&& S(S(P))-6PS(P)+4P(\lambda +P^2)-S(\lambda)=0. \label{T1equation}
	\end{eqnarray}
	\end{lemma}
\begin{proof}
	\begin{eqnarray}
	\nonumber \tilde{S}(\tilde{\lambda}) &=& S(\lambda +a)-2P\Delta (\lambda +a) \\
	\nonumber &=& S(\lambda)+2PS(P)-S(S(P))-4P(\lambda +a)\\
	\nonumber &=& -\left(S(S(P)-2PS(P)+4P(\lambda +P^2-S(P))-S(\lambda)\right)\\
	\nonumber &=& -\left(S(S(P)-6PS(P)+4P(\lambda +P^2)-S(\lambda)\right).
	\end{eqnarray}
\end{proof}
We summarize all the previous observations:
\begin{proposition} \label{Propsummary}
	Consider a spray $S$ that satisfies  $[\nabla \Phi,\Phi]=0$ on $D_{\lambda}$ for an eigenfunction $\lambda$ of $\Phi$. Consider a projective change by $P$ that satisfies equation (\ref{T1equation}). Then, $\tilde{S}$ meets the bracket property if and only if
	\begin{equation}\label{T2}
		 (\nabla b)+3Pb-(\lambda+a)\vlift{d}P=0 \quad \text{on} \quad D_\lambda.
	\end{equation}
\end{proposition}

In most situations, however, the starting spray does not satisfy  $[\nabla \Phi,\Phi]=0$ on $D_{\lambda}$
. We give a further characterization, in case the spray is isotropic.
\begin{definition}
	A spray $S$ is called isotropic if its Jacobi endomorphism is of the form
\[
\Phi (Y) = \lambda Y +c(Y)\T,\qquad\qquad \forall Y \in \vectorfields \tau,
\]
for some one-form $c$.

\end{definition}

The matrix of an isotropic spray in a standard basis is of the form

\[
\begin{bmatrix} 
\lambda & 0 & \dots & 0\\
0 & \lambda & \dots & 0 \\
\vdots & \vdots & \ddots & \vdots\\
0 &   0 &   \dots & \lambda & 
\end{bmatrix}
+
	\begin{bmatrix} 
c_1 {\dot x}_1 & c_2 {\dot x}_1 & \dots & c_n {\dot x}_1\\
c_1  {\dot x}_2 & c_2 {\dot x}_2 & \dots & c_n {\dot q}_2 \\
\vdots & \vdots & \ddots & \vdots\\
c_1 {\dot x}_n &   c_2 {\dot x}_n &   \dots & c_n {\dot x}_n & 
\end{bmatrix}
=
\begin{bmatrix}
\lambda + c_1 {\dot x}_1 & c_2 {\dot x}_1 & \dots & c_n {\dot x}_1\\
c_1  {\dot x}_2 & \lambda +c_2 {\dot x}_2 & \dots & c_n {\dot x}_2 \\
\vdots & \vdots & \ddots & \vdots\\
c_1 {\dot x}_n &   c_2 {\dot x}_n &   \dots & \lambda + c_n {\dot x}_n
\end{bmatrix}.
\]
It is clear that it has only two eigenvalues, zero with multiplicity $1$, and $\lambda$ with multiplicity $(n-1)$. An isotropic spray is therefore always diagonalizable. An isotropic spray remains isotropic after a projective change:
\[
\tilde{\Phi}(Y)=\lambda Y+c(Y)\T +aY+b(Y)\T.
\]
The nonzero eigenvalue of $\tilde\Phi$ is now $\lambda +a$ and its corresponding one-form is $c+b$. As a consequence, the projectively changed $\Phi$ is also diagonalizable and a vector $Y$ is an eigenvector of $\tilde{\Phi}$ corresponding to the nonzero eigenfunction $\lambda +a$ if and only if $(b+c)(Y)=0$.

\begin{proposition}
	Let $S$ be an isotropic spray, and consider a projective change $\tilde{S}:=S-2P\Delta$. The bracket property holds true for $\tilde{S}$ if 
\begin{equation} \label{T2i}
\nabla (c+b) +\left(P-\frac{S(\lambda +a)}{\lambda +a}\right)b +(\lambda+a)d^VP=0 \qquad \mbox{on $D_\lambda$}.
\end{equation}	\end{proposition}
\begin{proof} When we plug in the explicit form of $\Phi$ into expression (\ref{T0equation}) of Proposition~\ref{T0}, we find 
\[(-c-b)(\nabla X) +S(b(X))-b(X)\frac{S(\lambda +a)}{\lambda +a} +(\lambda+a)\vlift{X}(P)+b(X)P=0
\]
holds true for all $X\in D_\lambda$. Since $c(\T) = 0$, this reduces to the expression in the statement. \end{proof}

\section{Some worked-out examples}\label{ex1}

{\bf Example.} Consider the following spray on $\R^3$
\[
S=\dot{x}\frac{\partial}{\partial x}+\dot{y}\frac{\partial}{\partial y}+\dot{z}\frac{\partial}{\partial z}+(\dot{y}\dot{z}+\dot{x}\dot{z})\frac{\partial}{\partial\dot{x}}+(-\dot{x}\dot{z}+\dot{y}\dot{z})\frac{\partial}{\partial\dot{y}}+\dot{z}^2\frac{\partial}{\partial\dot{z}}.
\]
One may calculate that the function $\lambda=\frac{\dot{z}^2}{2}$ is an eigenfunction of the Jacobi endomorphism $\Phi$, but that it does not remain constant along solutions, since $S(\lambda)\neq 0$. However, the spray satisfies $[\nabla \Phi,\Phi]=0$, so we are in the situation of Proposition~\ref{Propsummary}. Consider a geodesic with $\dot z\neq 0$. As an ansatz, we wish to find a constant $A\in\R$, such that the projective change by factor
\[
P=A\cdot \dot{z}
\]
satisfies the equation (\ref{T1equation}) in Proposition~\ref{T1}. One finds that this equation takes the form
\[
\dot{z}^3(2A-1)(2A^2-2A+1)=0
\]
with only real solution $A=\frac{1}{2}$. This corresponds with the projective factor $P=\frac{1}{2}\dot{z}$, which also satisfies the equation (\ref{T2}). Now, the non-zero eigenfunction of the modified spray $\tilde S$ is $\tilde\lambda = \lambda + a = \frac{\dot{z}^2}{2} - \frac{\dot{z}^2}{4}  = \frac{\dot{z}^2}{4}$. It   remains constant along geodesics and, since it is positive, Proposition~\ref{basicprop}  guarantees the existence of conjugate points along the geodesic.

{\bf Example.} \textit{'Shen's circles'.} Consider the following spray on $\R^2$
\[
S=\dot{x}\frac{\partial}{\partial x}+\dot{y}\frac{\partial}{\partial y}-2\tau\dot{y}\sqrt{\dot{x}^2+\dot{y}^2}\frac{\partial}{\partial\dot{x}}+2\tau\dot{x}\sqrt{\dot{x}^2+\dot{y}^2}\frac{\partial}{\partial\dot{y}},
\]
where $\tau$ is an arbitrary positive number (see \cite{crampin2012multiplier,shen2013differential}). Since the spray does not satisfy $[\nabla \Phi,\Phi]=0$ we need a solution of the equation (\ref{T2i}). One may verify that the projective factor
\[
P=\frac{\tau^2\sqrt{\dot{x}^2+\dot{y}^2}(x\dot{x}+y\dot{y})}{\tau(x  \dot{y}- y\dot{x})-\sqrt{\dot{x}^2+\dot{y}^2}}
\]
satisfies equation (\ref{T2i}) and that the corresponding new spray, 
\begin{eqnarray*}
\tilde{S}&=&\dot{x}\frac{\partial}{\partial x}+\dot{y}\frac{\partial}{\partial y}+\frac{2\tau(\tau x (\dot{x}^2+\dot{y}^2) -\dot{y}\sqrt{\dot{x}^2+\dot{y}^2})\sqrt{\dot{x}^2+\dot{y}^2}}{\tau y\dot{x}-\tau x \dot{y}+\sqrt{\dot{x}^2+\dot{y}^2}}\frac{\partial}{\partial\dot{x}} \\&& \hspace*{4cm}+\frac{2\tau (\tau y (\dot{x}^2 +\dot{y}^2)+\dot{x}\sqrt{\dot{x}^2+\dot{y}^2})\sqrt{\dot{x}^2+\dot{y}^2}}{\tau y \dot{x}- \tau x \dot{y}+\sqrt{\dot{x}^2+\dot{y}^2}}\frac{\partial}{\partial\dot{y}},
\end{eqnarray*}
is, in fact, the geodesic spray of the Finsler function 
\[
F_\tau=\frac{\sqrt{\dot{x}^2+\dot{y}^2}+\tau(y\dot{x}-x\dot{y})}{2},
\]
which is given in the class (A) in the classification of \cite{crampin2013class} we had mentioned in the Introduction. Its geodesics are circles of radius $\frac{1}{2\tau}$. We would like to use Proposition \ref{basicprop} to find conjugate points along the geodesic $c(t)=\frac{1}{2\tau}(\cos(t),\sin(t))$.  Along the geodesic, it so happens that the Jacobi endomorphism does not depend on the parameter $\tau$. Its nonzero eigenfunction has always the constant value 1. Therefore we conclude that the points $c\big(\frac{k\pi}{\sqrt{\lambda}}\big)=c(k\pi)$ are conjugate to $c(0)=(\frac{1}{2\tau},0)$ along $c(t)$.

The conjugate points at time $k\pi$, with $k$ even, are identical with $c(0)$. Since the geodesics emanating from that point are circles of a fixed radius, they will all go through the same point again. In other words, there exists a geodesic variation with a constant curve as a transversal at times $k\pi$, so that the induced Jacobi field vanishes at these times.

We would like to find the Jacobi field that vanishes at times $k\pi$, for any $k$. From the proof of Proposition~\ref{basicprop} we know that the Jacobi field we are interested in is of the type $J(t)=V(t)\sin(t)$, where $V(t)$ is a parallel eigenvector field. This is a vector field along $c(t)$ that satisfies $\nabla_cV(t)=0$ and $\Phi_cV(t)=1\cdot V(t)$. Since in the current context $V(t)=\cos(t)\frac{\partial}{\partial x}\rvert_{c(t)}+\sin(t)\frac{\partial}{\partial, y}\rvert_{c(t)}$ satisfies the above mentioned two conditions, a Jacobi field is given by
\[
J(t)=\sin(t)\cos(t)\frac{\partial}{\partial x}\bigg\rvert_{c(t)}+\sin(t)^2\frac{\partial}{\partial y}\bigg\rvert_{c(t)}.
\]
This field has the property, that $\nabla_cJ(t)=\frac{\partial}{\partial x}\rvert_{c(t)}$. With the help of Maple we have plotted below some of  the transversals, corresponding to the  variation
\[
\gamma(u,s)=\exp(s(v+uw))
\]
with $v=\dot{c}(0)$ and $w=\nabla_cJ(0)$ and $\exp$ the exponential mapping of the spray $\tilde S$. The plot shows the case $\tau=1$.
One can observe that the vanishing of the velocity of the curves corresponds to the cusp points. 
\begin{center}
	\begin{tabular}{cc}
		\begin{minipage}{7cm}
			\begin{center}
				\includegraphics[height=7cm]{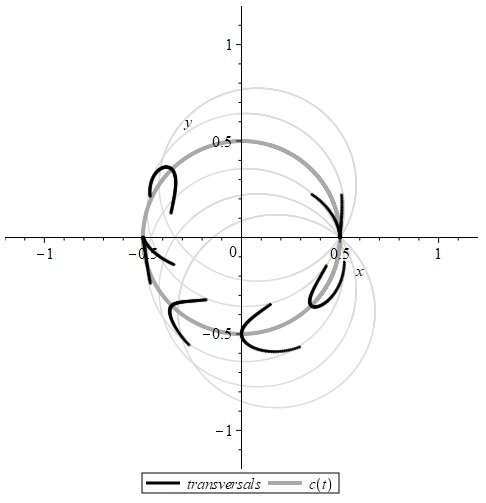}\\ Transversal curves of the geodesic variation $\gamma(u,s)$. 
			\end{center}
		\end{minipage}	
	\end{tabular}
\end{center}

\section{Cut points and conjugate points}\label{ex2} 

In the examples of Section~\ref{projclass}  and Section~\ref{ex1} we have found Jacobi fields with the help of a non-vanishing parallel eigenvector field $V(t)$. This   vector field was defined for all parameter values along the whole curve. Because of the local nature of our methods, it can happen, though, that $V(t)$ is not defined on a big enough domain and therefore we have to explore other methods to find the conjugate points. For this purpose we will recall the definition of cut points and their connection to conjugate points in the context of Finsler geometry. We will need the assumption that the geodesics of the underlying Finsler manifold are indefinitely forward extendible; such spaces are called forward geodesically complete Finsler manifolds.

\begin{definition}
	Let $(M,F)$ be a forward geodesically complete Finsler manifold and $c(t)$ be a geodesic. We define the  conjugate value \textbf{con(c)} and the  cut value \textbf{cut(c)} of $c(t)$ as
	\begin{itemize}
		\item \textbf{con(c)} = $\sup \{r\in\R: \text{no point c(t) with} \quad0\leq t \leq r  \quad \text{is conjugate to c(0)} \}$, \\
		\item \textbf{cut(c)} = $sup \{r\in\R : \text{ c(t) globally minimizes Finslerian length until c(r)} \}$.
	\end{itemize}
	The points on $c(t)$ correspoinding to these parameter values are called the first conjugate point and the cut point of $c(t)$.
\end{definition}
With a global minimizer we mean a geodesic, with the property that no other continious piecewise-differentiable curve exists with the same endpoints, but having shorter length. From \cite{BCS2} (Section~8) we recall a proposition about the connection between the cut and the first conjugate point.

\begin{proposition}\label{cutconjugate} 
	Let $(M,F)$ be a forward geodesically complete Finsler manifold and $c(t)$ be a geodesic. Then $\textbf{cut(c)} \leq \textbf{con(c)}$ and at least of the following two scenarios must hold
	\begin{itemize}
		\item[(1)]  $\textbf{cut(c)} \, = \, \textbf{con(c)}$
		\item[(2)] there exists two distinct geodesics of the same length connecting $c$ and its cut point.
	\end{itemize}
\end{proposition}
The first part of the theorem says that the first conjugate point can not appear before the cut point. From the second half we conclude that in case $\textbf{cut(c)} \neq \textbf{con(c)}$, there must exist two distinct geodesics of the same length connecting $c$ and its cut point. 

{\bf Example.} Let $\tilde S$ be the canonical spray  of the Finsler function
\[
F_\tau=\frac{\sqrt{\dot{x}^2+\dot{y}^2}+\tau(y\dot{x}-x\dot{y})}{2},
\]
corresponding to the class (A) in the Introduction (and to the example of `Shen's circles' in Section~\ref{ex1}). We consider the geodesic $c(t)=\frac{1}{2}(\cos(t),\sin(t))$ at the parameter value $t=\pi$. Since the geodesics of this spray are circles of radius $\frac{1}{2}$ (parametrized counter-clockwise), there is a unique geodesic connecting $c(0)$ to $c(r)$ if $r=\pi$ and there are two if $r\neq\pi$. Let $\epsilon$ be a small positive number. The length of $c(t)$ between $c(0)$ and $c(\pi +\epsilon)$ equals $\int_{0}^{\pi+\epsilon}F(\dot{c}(t))dt=\int_{0}^{\pi+\epsilon}\frac{1}{8}dt= \frac{\pi+\epsilon}{8}$, while the length of the second geodesic, say $\gamma_1(t)$ is $\frac{2\sin(\epsilon)+\pi-\epsilon}{8}$.

\begin{center}
	\begin{tabular}{cc}
		\begin{minipage}{7cm}
			\begin{center}
				\includegraphics[height=7cm]{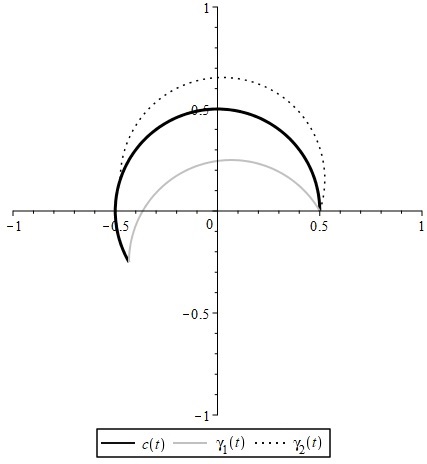}
			\end{center}
		\end{minipage}	
	\end{tabular}
\end{center}
Given that the inequality $\epsilon<\sin(\epsilon)$ holds true for positive $\epsilon$ we conclude that $c(t)$ fails to be a global minimizer after the parameter value $\pi$. The exact opposite happens if we consider the parameter value $t=\pi-\epsilon$. In this case $c(t)$ will be a shorter curve with length $\frac{\pi-\epsilon}{8}$ as opposed to the other geodesic, say $\gamma_2(t)$, whose length is greater by $\frac{\sin(\epsilon)}{4}$ than that of $c(t)$. As a conclusion the cut point of $c(t)$ is $c(\pi)$. Since $c(0)$ and $c(\pi)$ can not be connected with two distinct geodesics, case (1) of Proposition \ref{cutconjugate}. applies, that is, the cut point $c(\pi)$ is the first conjugate point, at the same time.

{\bf Example.} In this  example only geometric considerations help us to find conjugate points. Consider the canonical spray of the Finsler function  
\[
F=\frac{\sqrt{\dot{x}^2+\dot{y}^2}+y\dot{x}-x\dot{y}}{2(1-x^2-y^2)}.
\]
This Finsler function corrresponds to the  class (C) of the classification of the Introduction, with $\tau=1$. Its
 geodesics are so-called horocycles. These are circles (again, parametrized counter-clockwise) which are located inside the unit circle and which are tangent to it. We choose the geodesic  $c(t)=\frac{1}{2}(\cos(t-\frac{\pi}{2}),\sin(t-\frac{\pi}{2})+1)$. One may verify that the canonical spray is such that its Jacobi endomorphism $\Phi$ satisfies the bracket property $[\nabla \Phi,\Phi]=0$. Moreover, the nonzero eigenfunction of $\Phi$ has the constant value $\frac{1}{4}$ along $c(t)$. 

We are therefore in a situation where our methods apply. The problem, however, is that the parallel eigenvector field $V(t)$ has the form
\[
V(t)=\frac{1}{\sqrt{\sin(t)+1}}\left((-\sin(t)-1)\frac{\partial}{\partial x}\bigg\rvert_{c(t)}+\cos(t)\frac{\partial}{\partial y}\bigg\rvert_{c(t)}\right),
\]
and that it can not be extended further then $\frac{3\pi}{2}$  along the curve $c(t)$. Our methods would require the vector field to be defined at least on the domain $[0,\frac{\pi}{\sqrt{\frac{1}{4}}}]=[0,2\pi]$. We have therefore left no explicit expression of the Jacobi field. Nevertheless, after some calculation we can draw a similair conclusion as in the previous example. 
\begin{center}
	\begin{tabular}{cc}
		\begin{minipage}{7cm}
			\begin{center}
				\includegraphics[height=7cm]{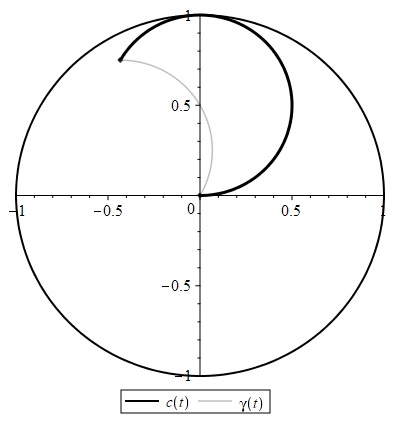}\\ The longer and the shorter geodesic connecting $c(0)$ and $c(\pi +\epsilon)$.
			\end{center}
		\end{minipage}	
	\end{tabular}
\end{center}

Now, the length of the longer geodesic $c(t)$ can readily calculated to be $\frac{\pi+\epsilon}{4}$, while the shorter geodesic $\gamma (t)$ will be of length $\frac{\pi-\epsilon}{4}$. We may therefore conclude that the cut point $c(\pi)$ again agrees with the first conjugate point.

\section*{Acknowledgments}

TM thanks the Research Foundation -- Flanders (FWO) for its support through Research Grant 1510818N.

\end{document}